\documentclass[sn-mathphys-num]{sn-jnl}

\usepackage{graphicx}%
\usepackage{multirow}%
\usepackage{amsmath,amssymb,amsfonts}%
\usepackage{amsthm}%
\usepackage{mathrsfs}%
\usepackage[title]{appendix}%
\usepackage{xcolor}%
\usepackage{textcomp}%
\usepackage{manyfoot}%
\usepackage{booktabs}%
\usepackage{algorithm}%
\usepackage{algorithmicx}%
\usepackage{algpseudocode}%
\usepackage{listings}%

\theoremstyle{thmstyleone}%
\newtheorem{theorem}{Theorem}[section]
\theoremstyle{thmstyletwo}%
\newtheorem{remark}{Remark}[section]%

\theoremstyle{thmstylethree}%

\newtheorem{lemma}{Lemma}[section]

\newtheorem{corollary}{Corollary}[section]

\begin{document}

\title[]{Long-term behaviour of symmetric partitioned linear multistep methods I.  Global error and conservation of invariants}


\author*[1]{\fnm{Bego\~na} \sur{Cano}}\email{bcano@uva.es}

\author[1]{\fnm{\'Angel} \sur{Dur\'an}}\email{angeldm@uva.es}

\author[1]{\fnm{Melquiades} \sur{Rodr\'\i guez}}\email{melquiades.rodriguez@uva.es}

\affil[1]{\orgdiv{Applied Mathematics Department}, \orgname{University of Valladolid}, \orgaddress{\street{P/ Belen, 7}, \city{Valladolid}, \postcode{47011},  \country{Spain}}}


\abstract{In this paper an asymptotic expansion of the global error on the stepsize for partitioned linear multistep methods is proved. This provides a tool to analyse the behaviour of these integrators with respect to error growth with time and conservation of invariants. In particular, symmetric partitioned linear multistep methods with no common roots in their first characteristic polynomials, except unity, appear as efficient methods to approximate non-separable Hamiltonian systems since they can be explicit and show good long term behaviour at the same time. As a case study, a thorough analysis is given for small oscillations of the double pendulum problem, which is illustrated by numerical experiments.}




\maketitle

\section{Introduction}\label{sec1}
The advantages that geometric integrators offer when integrating ordinary differential problems with a certain structure are well-known \cite{HLW_geom,LR,SSC}. In particular, symplecticity and symmetry are always desirable properties of the integrators in order to emulate the qualitative behaviour of Hamiltonian and reversible systems and even to conserve or better approximate invariants of the problem.

Nevertheless, when integrating general first-order systems, symplectic Runge-Kutta methods do necessarily need to be implicit. Just for special types of systems, explicit symplectic Runge-Kutta-type methods can be constructed. That is the case of second-order systems where the first-derivative does not turn up, for which there exist explicit symplectic Runge-Kutta-Nystr\"om methods and separable  problems for which partitioned Runge-Kutta methods can be implemented in an explicit way. Even for some very special classes of nonseparable Hamiltonian problems, there exist some explicit symplectic integrators especifically constructed for them, like those in \cite{B,C,MQ,T}. However, for general nonseparable Hamiltonian problems, we just know of the quite recent semiexplicit symplectic integrators in \cite{JO, O}, which preserve quadratic invariants but are computationally not clearly more efficient than the classical implicit Gauss methods.

On the other hand, in spite of the fact that linear multistep methods (LMMs) cannot be symplectic \cite{ESS}, the ones which are symmetric satisfy that their underlying one-step method is conjugate-symplectic \cite{Hairer}. This is good since LMMs can be explicit and symmetric at the same time and a lot of effort has been made in the literature to develop and analyse these methods \cite{CD1, CD2, CSS, HL, QT}. However, the advantageous theoretical results which have been obtained for long-term integration just concern again either the integration of second-order differential systems with no first derivatives with symmetric LMMs especially designed for them (and with no double roots in their first characteristic polynomial \cite{CSS,HL}) either the integration of separable Hamiltonian systems with symmetric partitioned LMMs (PLMMs) where the first characteristic polynomials are not allowed to have common roots except for that which gives consistency \cite{CH}.

In the case of LMMs for special 2nd-order systems, the analysis was firstly done in terms of asymptotic expansions of the global error and the conclusion was that, in some problems, like planetary motion ones, the propagation of error just in some specific direction led to a very slow error growth with time compared with other non-symmetric LMMs. On the contrary, general symmetric LMMs applied to the underlying first-order systems were shown to lead to exponential error growth with time, much worse than non-symmetric LMMs. Later on, an analysis on the same type of methods for 2nd-order systems was performed using modified differential equations and very advantageos results on the approximation of invariants were observed.

In the case of PLMMs, the analysis was done in terms of modified differential equations and the conclusion was that, when the differential system is reversible, the modified differential equation which the smooth part of the numerical solution satisfies, is also reversible when the method is symmetric. Furthermore, when the Hamiltonian is separable with quadratic kinetic energy, the modified equation is Hamiltonian. On the other hand, in separable Hamiltonian problems, the components  associated to the parasitic roots are shown to remain bounded  and small. However, no justification is given in \cite{CH} on the behaviour of symmetric PLMMs  with respect to invariants when the Hamiltonian is not quadratic and even less when the Hamiltonian is not separable.

The aim of this paper is thus to explain the good long term behaviour of symmetric PLMMs when integrating ordinary differential systems which do not necessarily correspond to separable Hamiltonian problems. In case this problem has invariants, an error analysis will be pursued. For that, we will focus on the asymptotic expansion of the global error, which is thoroughly studied for general PLMMs. We will illustrate the conclusions which can be drawn with our analysis in a particular example, that of a double pendulum, which corresponds to a non-separable Hamiltonian problem. We will see that in the case of small oscillations, our analysis allows to justify the good behaviour of symmetric PLMMs (with no common roots in their first characteristic polynomial except unity) against symmetric LMMs or non-symmetric LMMs.

We advance that this advantageous behaviour can be seen and justified in other problems in the context of time integration of partial differential problems \cite{CD3}.

The paper is structured as follows: In Section 2, PLMMs are introduced and an asymptotic expansion for the global error on the stepsize is described in detail. For the sake of clarity, the proof of that is given in Appendix \ref{secA1}. In Section 3, the growth of error with time of the coefficients of that asymptotic expansion is analysed. For that, symmetry and the accuracy of the starting approximation are taken into account. These results are used in Section 4 in order to understand the behaviour in the numerical integration of those invariant quantities that the differential system may admit. Hamitonian systems and symmetric methods are again given particular consideration. Finally, a thorough analysis is performed for the case of small oscillations of a double pendulum, justifying why certain symmetric PLMMs perform so well in comparison with other non-symmetric PLMMs or symmetric non-partitioned LMMs. We end with some concluding remarks.

\section{Asymptotic expansion of the global error}\label{sec2}

We firstly notice that any initial value problem (IVP) of dimension $n \ge 2$ can always be written in the form
\begin{eqnarray}
\dot{p}(t)&=& f(p(t),q(t)), \nonumber \\
\dot{q}(t)&=& g(p(t),q(t)), \quad t_{0}\leq t\leq T,\label{system} \\
p(t_0)&=&p_0, \nonumber \\
q(t_0)&=&q_0. \nonumber
\end{eqnarray}
In (\ref{system}), we will assume that $T>t_0$, $p=p(t)\in \mathbb{R}^d, q=q(t)\in \mathbb{R}^{n-d},  1 \le d \le n-1$,  where $f=f(p,q)$ and $g=g(p,q)$ are smooth enough so that existence and uniqueness of solutions of (\ref{system}) are guaranteed as well as $C^{\infty}$-regularity of the solution $(p(t),q(t))$. (Less regularity may be sufficient for the foregoing theorems but this is assumed for the sake of simplicity.)

For the numerical approximation of (\ref{system}), we consider PLMMs, which are determined by a starting procedure and the following difference equations
\begin{eqnarray}
\rho_p(E) p_n&=& h \sigma_p (E) f(p_n,q_n), \nonumber \\
\rho_q(E) q_n&=& h \sigma_q (E) g(p_n,q_n),\quad h>0, \label{eq12}
\end{eqnarray}
with irreducible real generating polynomials $(\rho_p, \sigma_p)$ and $(\rho_q, \sigma_q)$, respectively.

Concerning (\ref{eq12}), we will make the following nonrestrictive hypotheses:
\begin{itemize}
\item[(i)] If $\mbox{deg}(\rho_p)=k_p$, $\mbox{deg}(\rho_q)=k_q$, then (\ref{eq12}) uniquely determines $p_{n+k_p}$ and $q_{n+k_q}$ from given $p_n,\dots,p_{n+k_p-1}$ and $q_n,\dots,q_{n+k_q-1}$.
\item[(ii)] Both methods are zero-stable and have the same order $r$, and the corresponding local truncation error can be written as
\begin{eqnarray}
\rho_p(E) y(t_n)-h \sigma_p(E) \dot{y}(t_n)= \sigma_p(E) \big( \sum_{j=r}^{2r-1} c_{j,p} h^{j+1} y^{(j+1)}(t_n)\big) +O(h^{2r+1}), \nonumber \\
\rho_q(E) y(t_n)-h \sigma_q(E) \dot{y}(t_n)= \sigma_q(E) \big( \sum_{j=r}^{2r-1} c_{j,q} h^{j+1} y^{(j+1)}(t_n)\big)+O(h^{2r+1}), \label{LTE}
\end{eqnarray}
where $t_{n}=nh, h>0, n=0,1,\ldots$
\item[(iii)] The  starting values $p_\nu$ ($\nu=0,1,\dots,k_p-1$), $q_\nu$ ($\nu=0,1,\dots,k_q-1$) are assumed to satisfy
\begin{eqnarray}
p(t_\nu)-p_\nu&=&O(h^r), \quad \nu=0,1,\dots,k_p-1,\nonumber\\
q(t_\nu)-q_\nu&=&O(h^r),\quad \nu=0,1,\dots,k_q-1.\label{bcad0a}
\end{eqnarray}
\item[(iv)] Let us denote by $\{x_i\}_{i=1}^m $ the common roots of modulus one of $\rho_p$ and $\rho_q$ (with $x_1=1$). We will also denote by $x_{i,p} (i=m+1,\dots,k_p')$ the unitary roots of $\rho_p$ which are not roots of $\rho_q$ and by $x_{i,q} (i=m+1,\dots, k_q')$ the unitary roots of $\rho_q$ which are not roots of $\rho_p$. Notice that all these roots are single because of the assumed zero-stability of both methods, see  (ii). According to \cite{R}, that assures zero-stability of the whole partitioned linear multistep method.
\item[(v)] We use the notation
\begin{eqnarray}
\rho_{\alpha,i}(x)&=&\rho_\alpha(x_i x), \quad \rho_{\alpha,i,\beta}(x)=\rho_\alpha(x_{i,\beta} x), \quad \alpha,\beta\in \{p,q\},\label{bcad0b}\\
\sigma_{\alpha,i}(x)&=&\sigma_\alpha(x_i x), \quad  \sigma_{\alpha,i,\beta}(x)=\sigma_\alpha(x_{i,\beta} x),\label{bcad0c}
\end{eqnarray}
and let
\begin{eqnarray}
\lambda_{p,i}= \frac{\sigma_p(x_i)}{x_i \rho_p'(x_i)}, \quad \lambda_{q,i}= \frac{\sigma_q(x_i)}{x_i \rho_q'(x_i)}. \label{gp}
\end{eqnarray}
As the methods $(\rho_{\alpha,i}, \sigma_{\alpha,i}/\lambda_{\alpha,i}), \alpha=p,q,$
are consistent, we assume that the associated local truncation errors can be written as, cf. (\ref{LTE}),
\begin{eqnarray}
\rho_{p,i}(E) y(t_n)-\frac{h}{\lambda_{p,i}} \sigma_{p,i}(E) \dot{y}(t_n)&=& \frac{1}{\lambda_{p,i}}\sigma_{p,i}(E) \big( \sum_{j=1}^{r-1} c_{j,p}^{(i)} h^{j+1} y^{(j+1)}(t_n)\big)
\nonumber\\ &&
 +O(h^{r+1}), \nonumber \\
\rho_{q,i}(E) y(t_n)-\frac{h}{\lambda_{q,i}} \sigma_{q,i}(E) \dot{y}(t_n)&=& \frac{1}{\lambda_{q,i}}\sigma_{q,i}(E) \big( \sum_{j=1}^{r-1} c_{j,q}^{(i)} h^{j+1} y^{(j+1)}(t_n)\big)
\nonumber\\ &&
+O(h^{r+1}), \label{LTEb}
\end{eqnarray}
for some coefficients $c_{j,\alpha}^{(j)}$, $\alpha=p,q$.
\end{itemize}
The first aim is analyzing the global error which turns up when using the numerical method (\ref{eq12}) to approximate the exact solution $(p(t),q(t))$ of (\ref{system}). We have the following result on an asymptotic expansion on the stepsize $h$, the proof of which, for the sake of clarity, is in Appendix \ref{secA1}.

\begin{theorem} \label{th1}
For fixed $t_{n}=nh, n=1,\ldots$, and under conditions (i)-(v), there are smooth functions  $e_{j,i,\alpha}, e_{j,i,\alpha,\beta}$, with $\alpha,\beta \in \{p,q \}$, such that, as $h\rightarrow 0$,
\begin{eqnarray}
p_n-p(t_n)&=&\sum_{j=r}^{2r-1} h^j \big[\sum_{i=1}^m x_i^n e_{j,i,p}(t_n) +\sum_{i=m+1}^{k_p'} x_{i,p}^n e_{j,i,pp}(t_n)+ \sum_{i=m+1}^{k_q'} x_{i,q}^n e_{j,i,pq}(t_n)\big] \nonumber \\
&&+O(h^{2r}), \nonumber \\
q_n-q(t_n)&=&\sum_{j=r}^{2r-1} h^j \big[\sum_{i=1}^m x_i^n e_{j,i,q}(t_n) +\sum_{i=m+1}^{k_p'} x_{i,p}^n e_{j,i,qp}(t_n)+ \sum_{i=m+1}^{k_q'} x_{i,q}^n e_{j,i,qq}(t_n)\big] \nonumber \\
&&+O(h^{2r}). \label{ae}
\end{eqnarray}
Here, the functions $e_{j,1,\alpha}$, with $\alpha \in \{p,q\}$,  $j=r,\ldots,2r-1$, satisfy
\begin{eqnarray}
\hspace{-0.35cm}\dot{\left( \begin{array}{c} e_{j,1,p}(t) \\ e_{j,1,q}(t) \end{array} \right)}&=& \left( \begin{array}{cc} f_p(p(t),q(t)) & f_q(p(t),q(t)) \\ g_p(p(t),q(t)) & g_q(p(t),q(t)) \end{array}
\right) \left(\begin{array}{c} e_{j,1,p}(t) \\ e_{j,1,q}(t) \end{array} \right)
- \left( \begin{array}{c} c_{j,p} p^{(j+1)}(t) \\ c_{j,q} q^{(j+1)}(t) \end{array} \right),
\label{ej1}
\end{eqnarray}
and the functions $e_{j,i,\alpha}$, with $\alpha \in \{p,q\}$, $i=2,\ldots,m$, $j=r,\ldots,2r-1$ are solutions of
\begin{eqnarray}
\hspace{-0.45cm}\dot{\left( \begin{array}{c} e_{j,i,p}(t) \\ e_{j,i,q}(t) \end{array} \right)}&=& \left( \begin{array}{cc} \lambda_{p,i} f_p(p(t),q(t)) & \lambda_{p,i} f_q(p(t),q(t)) \\ \lambda_{q,i} g_p(p(t),q(t)) & \lambda_{q,i} g_q(p(t),q(t)) \end{array}
\right) \left(\begin{array}{c} e_{j,i,p}(t) \\ e_{j,i,q}(t) \end{array} \right)
+ \left( \begin{array}{c} b_{j,i,p}(t) \\ b_{j,i,q}(t) \end{array} \right),
\label{eji}
\end{eqnarray}
where
\begin{eqnarray*}
b_{j,i,p}(t)=-\sum_{l=1}^{j-r} c_{l,p}^{(i)} e_{j-l,i,p}^{(l+1)}(t), \quad b_{j,i,q}(t)=-\sum_{l=1}^{j-r} c_{l,q}^{(i)} e_{j-l,i,q}^{(l+1)}(t).
\end{eqnarray*}
On the other hand, the functions $e_{j,i,\alpha\beta}$, with $\alpha,\beta \in \{p,q\}$, $i=m+1,\ldots,k_{\beta}'$, $j=r,\ldots,2r-1$, are determined in a recursive way
from the following differential systems. More precisely, for $j=r,r+1,r+2$, there holds
\begin{eqnarray}
e_{r,i,pq}(t)&=&e_{r,i,qp}(t)=0,
\label{erpqqp} \\
\dot{e}_{r,i,pp}(t)&=&\lambda_{p,i,p}f_p(p(t),q(t))e_{r,i,pp}(t), \qquad \lambda_{p,i,p}=\frac{\sigma_p(x_{i,p})}{x_{i,p}\rho_p'(x_{i,p})},
\label{eripp} \\
\dot{e}_{r,i,qq}(t)&=&\lambda_{q,i,q}g_q(p(t),q(t))e_{r,i,qq}(t), \qquad \lambda_{q,i,q}=\frac{\sigma_q(x_{i,q})}{x_{i,q}\rho_q'(x_{i,q})},
\label{eriqq} \\
e_{r+1,i,pq}(t)&=&\frac{\sigma_p(x_{i,q})}{\rho_p(x_{i,q})}f_q(p(t),q(t))e_{r,i,qq}(t), \label{er1pq} \\
e_{r+1,i,qp}(t)&=&\frac{\sigma_q(x_{i,p})}{\rho_q(x_{i,p})}g_p(p(t),q(t))e_{r,i,pp}(t), \label{er1qp} \\
\dot{e}_{r+1,i,pp}(t)&=&\lambda_{p,i,p} f_p(p(t),q(t)) e_{r+1,i,pp}(t)+b_{r+1,i,pp}(t), \label{er1pp} \\
\dot{e}_{r+1,i,qq}(t)&=&\lambda_{q,i,q} g_q(p(t),q(t)) e_{r+1,i,qq}(t)+b_{r+1,i,qq}(t), \label{er1qq}
\end{eqnarray}
where, dropping the argument $t$ for ease of notation,
\begin{eqnarray}
b_{r+1,i,pp}&=&\lambda_{p,i,p}  \frac{\sigma_q(x_{i,p})}{\rho_q(x_{i,p})}f_q(p,q) g_p(p,q) e_{r,i,pp}\nonumber \\
&&-\frac{1}{2}[x_{i,p} \frac{\rho_p''(x_{i,p})}{\rho_p'(x_{i,p})}+1]\lambda_{p,i,p} \frac{d}{dt}[f_p(p,q) e_{r,i,pp}]
+\frac{\sigma_p'(x_{i,p})}{\rho_p'(x_{i,p})}\frac{d}{dt}[f_p(p,q) e_{r,i,pp}],
\nonumber \\
b_{r+1,i,qq}&=&\lambda_{q,i,q}  \frac{\sigma_p(x_{i,q})}{\rho_p(x_{i,q})} g_p(p,q) f_q(p,q) e_{r,i,qq}
\nonumber \\
&&-\frac{1}{2}[x_{i,q}\frac{\rho_q''(x_{i,q})}{\rho_q'(x_{i,q})}+1]\lambda_{q,i,q} \frac{d}{dt}[g_q(p,q) e_{r,i,qq}]+\frac{\sigma_q'(x_{i,q})}{\rho_q'(x_{i,q})}\frac{d}{dt}[g_q(p,q) e_{r,i,qq}],
\nonumber
\end{eqnarray}
and
\begin{eqnarray}
e_{r+2,i,pq}&=&\frac{x_{i,q}}{\rho_p(x_{i,q})}[\sigma_p'(x_{i,q})-\frac{\rho_p'(x_{i,q})\sigma_p(x_{i,q})}{\rho_p(x_{i,q})}]\frac{d}{dt}[f_q(p,q) e_{r,i,qq})]
\nonumber \\
&&+\frac{\sigma_p(x_{i,q})}{\rho_p(x_{i,q})}f_q(p,q)[\frac{\sigma_p(x_{i,q})}{\rho_p(x_{i,q})}f_p(p,q)e_{r,i,qq}+e_{r+1,i,qq}]. \label{er2pq} \\
e_{r+2,i,qp}&=&\frac{x_{i,p}}{\rho_q(x_{i,p})}[\sigma_q'(x_{i,p})-\frac{\rho_q'(x_{i,p})\sigma_q(x_{i,p})}{\rho_q(x_{i,p})}]\frac{d}{dt}[g_p(p,q) e_{r,i,pp}]
\nonumber \\
&&+\frac{\sigma_q(x_{i,p})}{\rho_q(x_{i,p})}g_p(p,q)[\frac{\sigma_q(x_{i,p})}{\rho_q(x_{i,p})}g_q(p,q)e_{r,i,pp}+e_{r+1,i,pp}], \label{er2qp}
\\
\dot{e}_{r+2,i,pp}&=&\lambda_{p,i,p} f_p(p,q)e_{r+2,i,pp}+ b_{r+2,i,pp},
\label{er2ipp} \\
\dot{e}_{r+2,i,qq}&=&\lambda_{q,i,q} g_q(p,q)e_{r+2,i,qq}+ b_{r+2,i,qq},
\label{er2iqq}
\end{eqnarray}
with
\begin{eqnarray}
\lefteqn{b_{r+2,i,pp}= \lambda_{p,i,p} f_q(p,q) e_{r+2,i,qp}} \nonumber \\
&&+\frac{1}{6}[\frac{3[x_{i,p}\sigma_p''(x_{i,p})+\sigma_p'(x_{i,p})-x_{i,p}\rho_p''(x_{i,p})]-x_{i,p}^2 \rho_p'''(x_{i,p})}{\rho_p'(x_{i,p})}-1]\frac{d^2}{d t^2}[f_p(p,q) e_{r,i,pp}] \nonumber \\
&&+\bigg[\frac{\sigma_p'(x_{i,p})}{\rho_p'(x_{i,p})}-\frac{\lambda_{p,i,p}}{2}[x_{i,p}\frac{\rho_p''(x_{i,p})}{\rho_p'(x_{i,p})}+1]\bigg]\frac{d}{dt}[f_p(p,q) e_{r+1,i,pp}] \nonumber \\
&&+\frac{\sigma_p'(x_{i,p})}{\rho_p'(x_{i,p})}\frac{d}{dt}[f_q(p,q)e_{r+1,i,qp}]-\frac{\lambda_{p,i,p}}{2}[x_{i,p}\frac{\rho_p''(x_{i,p})}{\rho_p'(x_{i,p})}+1]
\dot{b}_{r+1,i,pp},\nonumber \\
\lefteqn{b_{r+2,i,qq}= \lambda_{q,i,q} g_p(p,q) e_{r+2,i,pq}} \nonumber \\
&&+\frac{1}{6}[\frac{3[x_{i,q}\sigma_q''(x_{i,q})+\sigma_q'(x_{i,q})-x_{i,q}\rho_q''(x_{i,q})]-x_{i,q}^2 \rho_q'''(x_{i,q})}{\rho_q'(x_{i,q})}-1]\frac{d^2}{d t^2}[g_q(p,q) e_{r,i,qq}] \nonumber \\
&&+\bigg[\frac{\sigma_q'(x_{i,q})}{\rho_q'(x_{i,q})}-\frac{\lambda_{q,i,q}}{2}[x_{i,q}\frac{\rho_q''(x_{i,q})}{\rho_q'(x_{i,q})}+1]\bigg]\frac{d}{dt}[g_q(p,q) e_{r+1,i,qq}] \nonumber \\
&&+\frac{\sigma_q'(x_{i,q})}{\rho_q'(x_{i,q})}\frac{d}{dt}[g_p(p,q)e_{r+1,i,pq}]-\frac{\lambda_{q,i,q}}{2}[x_{i,q}\frac{\rho_q''(x_{i,q})}{\rho_q'(x_{i,q})}+1]
\dot{b}_{r+1,i,qq}.\nonumber
\end{eqnarray}
The initial conditions of all systems (\ref{ej1}), (\ref{eji}), (\ref{eripp}), (\ref{eriqq}), (\ref{er1pp}), (\ref{er1qq}), (\ref{er2ipp}) and (\ref{er2iqq}) are determined by the starting procedure.
Moreover,  when
\begin{eqnarray}
p_\nu-p(t_\nu)&=&O(h^{r+1}), \quad \nu=0,1,\dots,k_p-1, \nonumber \\
 q_\nu-q(t_\nu)&=&O(h^{r+1}), \quad \nu=0,1,\dots,k_q-1, \nonumber
\end{eqnarray}
it happens that $e_{r,i,pp}(t_0)=e_{r,i,qq}(t_0)=0$. Because of that, in such a case,
$$
0=e_{r,i,pp}(t)=e_{r,i,qq}(t)=e_{r+1,i,pq}(t)=e_{r+1,i,qp}(t)=b_{r+1,i,pp}(t)=b_{r+1,i,qq}(t).$$
However, more accurate starting procedures do not lead to the annihilation of the starting values of the coefficients associated to higher powers of $h$ and so no further simplifications can be obtained.
\end{theorem}

\subsection{Symmetric methods}

When (\ref{eq12}) is symmetric (that is, when both $(\rho_p,\sigma_p)$ and $(\rho_q,\sigma_q)$ are symmetric) some remarks can be made.
\begin{enumerate}
\item[(1)]
Due to the zero-stability (see assumption (ii)), all the roots of $\rho_p$ and $\rho_q$ have unit modulus and are single (we remind that $\rho_{\alpha}(x)=-x^{k_{\alpha}}\rho_{\alpha}(1/x)$ for symmetric methods).
This implies that, in formula (\ref{ae}), $k_p'$ and $k_q'$ can be substituted by $k_p$ and $k_q$ and, in fact, revising the proof of Theorem \ref{th1}, the result is valid for $n\ge 0$.

\item[(2)] It is also well known that the growth parameters $\lambda_{p,i}$, $\lambda_{q,i}$, $\lambda_{p,i,p}$ and $\lambda_{q,i,q}$ in (\ref{gp}), (\ref{eripp}), (\ref{eriqq}) are real \cite{S}. However, $\sigma_p(x_{i,q})/\rho_p(x_{i,q})$ and $\sigma_q(x_{i,p})/\rho_q(x_{i,p})$ in (\ref{er1pq}) and (\ref{er1qp}) are purely imaginary. This is due to the fact that, for these methods, $\rho_{\alpha}(x)=-x^{k_{\alpha}} \rho_{\alpha}(1/x)$ and $\sigma_{\alpha}(x)=x^{k_{\alpha}} \sigma_{\alpha}(1/x)$, so, for example,
$$
\frac{\sigma_p(x_{i,q})}{\rho_p(x_{i,q})}=\frac{x_{i,q}^{k_p} \sigma_p(\bar{x}_{i,q})}{-x_{i,q}^{k_p}\rho_p(\bar{x}_{i,q})}=-\frac{\sigma_p(\bar{x}_{i,q})}{\rho_p(\bar{x}_{i,q})}.
$$
In fact, if $x_{i,q}=\bar{x}_{i,q}=-1$, this implies that $\sigma_p(x_{i,q})/\rho_p(x_{i,q})=0$.
Moreover, some other coefficients in (\ref{er1pp})-(\ref{er2iqq}) are complex, without necessarily being purely imaginary. What is true, in any case, is that, for complex conjugate roots, those coefficients are complex conjugates. Thus, in the end, the sum is real.
\end{enumerate}

\section{Error growth analysis}\label{sec3}
From the results obtained in Section \ref{sec2}, the time behaviour of the error is here discussed. We will distinguish the different terms of the error corresponding to the associated roots of $\rho_p$ and $\rho_q$.

We remind that perturbations on the initial conditions of the continuous system (\ref{system}) propagate in time in first approximation through the variational system
\begin{eqnarray}
\dot{\delta}(t)=\left( \begin{array}{cc} f_p(p(t),q(t)) & f_q(p(t),q(t)) \\ g_p(p(t),q(t)) & g_q(p(t),q(t)) \end{array} \right) \delta(t). \label{ecvar}
\end{eqnarray}
We will denote by $\mathcal{M}(t,s)$ the corresponding transition matrix, which satisfies
$$\delta(t)=\mathcal{M}(t,s) \delta(s), \mbox{ for }t \ge s \ge t_0.$$

\subsection{Coefficients associated to the root $x_1=1$}
It is clear that the coefficients $e_{j,1,p}(t)$ and $e_{j,1,q}(t)$ in (\ref{ae}), which satisfy (\ref{ej1}), can be written as
$$
\dot{\left( \begin{array}{c} e_{j,1,p}(t) \\ e_{j,1,q}(t) \end{array} \right)}= \mathcal{M}(t,t_0) \left( \begin{array}{c} e_{j,1,p} \\ e_{j,1,q} \end{array} \right) (t_0)+\int_{t_0}^t \mathcal{M}(t,s)
\left( \begin{array}{c} -c_{j,p} p^{(j+1)}(s) \\ -c_{j,q}q^{(j+1)}(s) \end{array} \right)ds,
$$
for the transition matrix $\mathcal{M}(t,s)$ associated to  (\ref{ecvar}). The behaviour of $\mathcal{M}(t,s)$ determines the behaviour of $e_{j,1,p}(t)$ and $e_{j,1,q}(t)$ with time. For example, if
\begin{eqnarray}
\|\mathcal{M}(t,s)\|\le C[1+ (t-s)^\gamma],
\label{cotaM}
\end{eqnarray}
for some matrix norm, some real value $\gamma\ge 0$ and some constant $C>0$ independent of $s$ and $t$, and it happens that the derivatives $(p^{(j+1)}(t),q^{(j+1)}(t))$ are bounded in $[t_0, t_0+T]$, it follows that  $e_{j,1,p}(t)$ and $e_{j,1,q}(t)$ behave as $O((t-t_0)^{\gamma+1})$ when $t$ grows.
 \subsubsection{Symmetric methods}
 It is well-known  that, with symmetric methods, $c_{j,p}=c_{j,q}=0$ for odd $j$ \cite{S}. Because of this, $e_{j,1,p}(t)$ and $e_{j,1,q}(t)$ will grow like $O((t-t_0)^\gamma)$ for odd $j$ and like $O((t-t_0)^{\gamma+1})$ for even $j$. As symmetric methods have even order, the dominant term will correspond to $O((t-t_0)^{\gamma+1} h^r)$.

 \subsection{Coefficients associated to the roots $x_{i}\,\, (i=2,\dots,m)$}
\label{coefcurd1}
 In the same way as it happens with non-partitioned LMMs \cite{CSS}, the homogeneous part of the differential systems in (\ref{eji}) are different from (\ref{ecvar}). Because of that, in many problems, those coefficients lead to exponential growth of error with time. (See \cite{CSS}, where symmetric methods for first-order differential systems are not recommended for this reason.) However, with PLMMs, symmetric methods can be constructed where $x_1=1$ is the only common root of both $\rho_p$ and $\rho_q$. Then, that type of behaviour can be avoided.

 \subsection{Coefficients associated to the non-common roots of unitary modulus}
 \label{coefncr}
 The following result states how the coefficients associated to the non-common roots of unitary modulus behave.




\begin{theorem}
Under conditions of Theorem \ref{th1}, when the transition matrices associated to (\ref{eripp}) and (\ref{eriqq}) are bounded, assuming enough regularity and that the components of the Jacobian of the vector field are bounded at the exact solution as well as their time derivatives, the terms of the error associated to the non-common roots of $\rho_p$ and $\rho_q$  behave as $O(h^r)$ where the constant in Landau notation is bounded for $t-t_0=O(h^{-1})$.

Moreover, if the starting values differ from the exact ones in  $O(h^{r+1})$, the terms of the error associated to the non-common roots of $\rho_p$ and $\rho_q$ will behave as $O(h^{r+1})$ where the constant in Landau notation is bounded for $t-t_0=O(h^{-1})$. Equivalently, they behave as $O(h^r)$ uniformly in time for $t-t_0=O(h^{-2})$.
\label{th3}
\end{theorem}
\begin{proof}
The boundedness of the transition matrices associated to (\ref{eripp}) and (\ref{eriqq}) implies that $e_{r,i,pp}$ and $e_{r,i,qq}$ are bounded. In addition, from  (\ref{er1pq}) and (\ref{er1qp}), if $f_q(p(t),q(t))$ and $g_p(p(t),q(t))$ are bounded with time, the same happens with $e_{r+1,i,pq}$ and $e_{r+1,i,qp}$, as well as with $b_{r+1,i,pp}$ and $b_{r+1,i,qq}$ if $f_p(p(t),q(t))$, $g_q(p(t),q(t))$ and their first time derivatives are also bounded. This implies, using (\ref{er1pp}) and (\ref{er1qq}), that $e_{r+1,i,pp}$  and $e_{r+1,i,qq}$ grow at most linearly with time. On the other hand, from (\ref{er2pq}) and (\ref{er2qp}), if the time derivatives of $f_q(p(t),q(t))$ and $g_p(p(t),q(t))$ are also bounded with time,  $e_{r+2,i,pq}$  and $e_{r+2,i,qp}$ grow at most linearly. The same linear growth will be observed for $b_{r+2,i,pp}$ and $b_{r+2,i,qq}$ if the second time derivatives of $f_p(p(t),q(t))$ and $g_q(p(t),q(t))$ are bounded. This, together with (\ref{er2ipp}) and (\ref{er2iqq}) imply that $e_{r+2,i,pp}$ and $e_{r+2,i,qq}$ grow at most quadratically. And the argument can proceed inductively for the rest of terms.

Note finally that, if the starting procedure has order $r+1$, from the end of Theorem \ref{th1}, it follows that
$$e_{r,i,pp}(t)=e_{r,i,qq}(t)=b_{r+1,i,pp}(t)=b_{r+1,i,qq}(t)=e_{r+1,pq}(t)=e_{r+1,qp}(t)=0.$$
Therefore,  $e_{r+1,i,pp}$ and $e_{r+1,i,qq}$ are bounded with time and  if $f_q(p(t),q(t))$ and $g_p(p(t),q(t))$ are also bounded with time, the same happens with $e_{r+2,i,pq}$ and $e_{r+2,i,qp}$. This implies that, if $f_p(p(t),q(t))$ and $g_q(p(t),q(t))$ and their first time derivatives are also bounded, so $b_{r+2,i,pp}$ and $b_{r+2,i,qq}$ are. Due to that, $e_{r+2,i,pp}$  and $e_{r+2,i,qq}$  grow at most linearly and proceeding inductively, $e_{r+3,i,pp}$ and $e_{r+3,i,qq}$ grow at most quadratically, etc.

\end{proof}

\begin{remark}
When the problem is separable, i.e.,
\begin{eqnarray}
f=f(q), \quad g=g(p),
\nonumber
\end{eqnarray}
then not only the transition matrices associated to (\ref{eripp}) and (\ref{eriqq}) are bounded, but, since $f_p=g_q=0$, they are actually the identity. This simplifies formulas (\ref{er1pp})-(\ref{er2iqq}), but the conclusion is the same than that of the non-separable case.
\end{remark}

\begin{remark}
As a comparison, we notice that the separable case is studied in detail in \cite{CH} and our conclusions are similar to those obtained there: For an accurate enough starting procedure,  the parasitic solution components are under control for $t-t_0=O(h^{-2})$ (see the end of Section 3.5 and the concluding remarks in \cite{CH} ).
\end{remark}

\section{Conservation of invariants}\label{sec4}
In this section we will study the behaviour of PLMMs with respect to the conserved quantities of (\ref{system}).
\subsection{Smooth part of the numerical solution}
From the asymptotic expansion of the global error (\ref{ae}) in Theorem \ref{th1}, we will consider the \lq smooth\rq\ part associated to the root $x_{1}=1$ and define
\begin{eqnarray}
\left( \begin{array}{c} p_h(t) \\ q_h(t) \end{array} \right) := \left( \begin{array}{c} p(t) \\ q(t) \end{array} \right)+ \sum_{j=r}^{2r-1} h^j \left( \begin{array}{c} e_{j,1,p}(t) \\ e_{j,1,q}(t) \end{array} \right),\label{bcad60}
\end{eqnarray}
where $(p(t),q(t))$ is the solution of (\ref{system}).
\begin{lemma}
Let $I$ be a smooth invariant of (\ref{system}). Then,
\begin{eqnarray}
\frac{d}{d t} I \left( \begin{array}{c} p_h(t) \\ q_h(t) \end{array} \right)= - \sum_{j=r}^{2r-1} h^j (\nabla I) \left( \begin{array}{c} p(t) \\ q(t) \end{array} \right)^T \left( \begin{array}{c} c_{j,p} p^{(j+1)}(t) \\ c_{j,q} q^{(j+1)} (t) \end{array} \right)+ O(h^{2r}). \nonumber
\end{eqnarray}
\label{Ismooth}
\end{lemma}
\begin{proof}
Let $\delta_0\in\mathbb{R}^{n}$ be an arbitrary initial perturbation. Then the exact solution of (\ref{system}) with initial condition $(p_0,q_0)^T+\delta_0$ can be written as  $(p(t),q(t))^T+\delta(t)$ where $\delta(t)\in\mathbb{R}^{n}, t>t_0$ satisfies in first order the differential system (\ref{ecvar}) with $\delta(t_0)=\delta_0$.
Since $I$ is invariant and
\begin{eqnarray*}
I\left((p(t),q(t))^T+\delta(t)\right)=I\left((p(t),q(t))^T\right)+\nabla I\left((p(t),q(t))\right)^T\delta(t)+O\left(||\delta(t)||^{2}\right),
\end{eqnarray*}
from (\ref{ecvar}) there holds
\begin{eqnarray}
\lefteqn{\hspace{-5cm}\frac{d}{dt}[ (\nabla I) \left( \begin{array}{c} p(t) \\ q(t) \end{array} \right)] ^T \delta(t) +  (\nabla I) \left( \begin{array}{c} p(t) \\ q(t) \end{array} \right)^T \left( \begin{array}{cc} f_p(p(t),q(t)) & f_q(p(t),q(t)) \\ g_p(p(t),q(t)) & g_q(p(t),q(t)) \end{array} \right) \delta(t)} \nonumber \\
 &&=O\left(||\delta(t)||^{2}\right),\label{bcad61b}
\end{eqnarray}
which is valid for any arbitrary perturbation $\delta(t)$. Applying (\ref{bcad61b}) with $\delta(t)=(p_{h}(t),q_{h}(t))^{T}-(p(t),q(t))^{T}$, with $(p_{h}(t),q_{h}(t))^{T}$ given by (\ref{bcad60}), and using (\ref{ej1}) we have
\begin{eqnarray}
\lefteqn{
\frac{d}{d t} I \left( \begin{array}{c} p_h(t) \\ q_h(t) \end{array} \right)= \frac{d}{d t} \bigg[ I \left( \begin{array}{c} p(t) \\ q(t) \end{array} \right)+\sum_{j=r}^{2r-1} h^j  (\nabla I) \left( \begin{array}{c} p(t) \\ q(t) \end{array} \right)^T \left( \begin{array}{c} e_{j,1,p}(t) \\ e_{j,1,q}(t) \end{array} \right) +O(h^{2r}) \bigg]
}
\nonumber
\\
 &=& \sum_{j=r}^{2r-1} h^j \bigg[ \frac{d}{dt} [(\nabla I) \left( \begin{array}{c} p(t) \\ q(t) \end{array} \right)]^T \left( \begin{array}{c} e_{j,1,p}(t) \\ e_{j,1,q}(t) \end{array} \right) + (\nabla I) \left( \begin{array}{c} p(t) \\ q(t) \end{array} \right)^T \left( \begin{array}{c} \dot{e}_{j,1,p}(t) \\ \dot{e}_{j,1,q}(t) \end{array} \right) \bigg]
 \nonumber \\
 &&+O(h^{2r})   \nonumber \\
 &=& \sum_{j=r}^{2r-1} h^j \bigg[ \frac{d}{dt} [(\nabla I) \left( \begin{array}{c} p(t) \\ q(t) \end{array} \right)]^T \left( \begin{array}{c} e_{j,1,p}(t) \\ e_{j,1,q}(t) \end{array} \right)
 \nonumber
 \\
&& \hspace{0.1cm}+ (\nabla I) \left( \begin{array}{c} p(t) \\ q(t) \end{array} \right)^T \big[ \left( \begin{array}{cc} f_p(p(t),q(t)) & f_q(p(t),q(t)) %
\\ g_p(p(t),q(t)) & g_q(p(t),q(t)) \end{array} \right) \left( \begin{array}{c} e_{j,1,p}(t) \\ e_{j,1,q}(t) \end{array} \right)- \left( \begin{array}{c}  c_{j,p} p^{(j+1)}(t) \\ c_{j,q} q^{(j+1)} (t) \end{array} \right)
 \big] \bigg] \nonumber \\
 &&+ O(h^{2r}) \nonumber \\
 &=&-\sum_{j=r}^{2r-1} h^j (\nabla I) \left( \begin{array}{c} p(t) \\ q(t) \end{array} \right)^T \left( \begin{array}{c}  c_{j,p} p^{(j+1)}(t) \\ c_{j,q} q^{(j+1)} (t) \end{array} \right) + O(h^{2r}). \nonumber
 \end{eqnarray}
\end{proof}
\begin{corollary}
Assume that (\ref{system}) is Hamiltonian, with
\begin{eqnarray*}
f(p,q)=-\nabla_q H(p,q),\quad g(p,q)=\nabla_p H(p,q),
\end{eqnarray*}
for some smooth Hamiltonian function $H(p,q)$. Then
\begin{eqnarray}
\frac{d}{d t} H \left( \begin{array}{c} p_h(t) \\ q_h(t) \end{array} \right)=  \sum_{j=r}^{2r-1} h^j [c_{j,q} \dot{p}(t)^T q^{(j+1)}(t)-c_{j,p} \dot{q}(t)^T p^{(j+1)}(t)]+O(h^{2r}).
\label{ham}
\end{eqnarray}

\label{corol_ham}
\end{corollary}

\subsubsection{Symmetric methods}
In this case, the above result can be simplified as follows:
\begin{theorem}
When both methods $(\rho_p,\sigma_p)$ and $(\rho_q,\sigma_q)$ are symmetric, then
\begin{eqnarray}
\lefteqn{H\left( \begin{array}{c} p_h(t) \\ q_h(t) \end{array} \right)-H\left( \begin{array}{c} p_h(t_0) \\ q_h(t_0) \end{array} \right)}\nonumber \\
&&=\sum_{k=r/2}^{r-1} h^{2k} \bigg[c_{2k,q}\sum_{l=1}^k (-1)^{l+1}[p^{(l)}(t)^T q^{(2k+1-l)}(t)-p^{(l)}(t_0)^T q^{(2k+1-l)}(t_0)] \nonumber \\
&& \hspace{2.2cm} -c_{2k,p}\sum_{l=1}^k (-1)^{l+1}[q^{(l)}(t)^T p^{(2k+1-l)}(t)-q^{(l)}(t_0)^T p^{(2k+1-l)}(t_0)] \bigg] \nonumber \\
&&+\sum_{k=r/2}^{r-1} h^{2k} (-1)^{k+2} (c_{2k,q}-c_{2k,p}) \int_{t_0}^t p^{(k+1)}(s)^T q^{(k+1)}(s)ds +O(t h^{2r}). \label{ham_smooth}
\end{eqnarray}

\label{rem_ham_sym}
\end{theorem}

\begin{proof}
Note first that, when the methods are symmetric then $c_{j,p}=c_{j,q}=0$ for odd $j$. Moreover, for even $j$ ($j=2k$),
\begin{eqnarray}
\dot{q}^T p^{(2k+1)}=\frac{d}{dt}\bigg[ \sum_{l=1}^k (-1)^{l+1} {q^{(l)}}^T p^{(2k+1-l)} \bigg]+(-1)^{k+2} {q^{(k+1)}}^T p^{(k+1)}, \nonumber \\
\dot{p}^T q^{(2k+1)}=\frac{d}{dt}\bigg[ \sum_{l=1}^k (-1)^{l+1} {p^{(l)}}^T q^{(2k+1-l)} \bigg]+(-1)^{k+2} {p^{(k+1)}}^T q^{(k+1)}. \label{difpq}
\end{eqnarray}
Applying (\ref{difpq}) to (\ref{ham}) leads to (\ref{ham_smooth}).
\end{proof}
\begin{remark}
When both methods $(\rho_p,\sigma_p)$ and $(\rho_q,\sigma_q)$ are the same and symmetric, then (\ref{difpq}) implies that the sum in (\ref{ham}) can be seen as a total differential. This is in close connection with the fact stated in \cite{CH} that, in such a case, there exists a modified Hamiltonian satisfied, except for a small remainder, by the smooth part of the solution (\ref{bcad60}). In such a case, however, it is more difficult to control the terms of the error associated to the common roots of $\rho_p$ snd $\rho_q$  different from $1$, as stated in Subsection \ref{coefncr}.  A strategy which is developed in \cite{CH} to avoid that is to consider two different symmetric methods for which $c_{r,q}=c_{r,p}$ and so that their first characteristic polynomials just have the common root $x_1=1$. In such a way, the remainder for the modified Hamiltonian of the smooth part is $O(h^{r+2})$. As distinct, in Section \ref{sec5} of this manuscript, we will center on getting a bound for the last integral in (\ref{ham_smooth}) for a case study, so that the advantageous behaviour of symmetric methods without satisfying the condition $c_{r,q}=c_{r,p}$ can be explained in detail.
\end{remark}
\subsection{Whole numerical solution when $m=1$}
We will now study the whole numerical solution. As stated in Subsection \ref{coefcurd1}, the coefficients of the error associated to the common unitary roots different from $x_1=1$ typically lead to an exponential error growth with time. Because of that,  we will assume from now on that such common roots do not exist, i.e. $m=1$. Then, we have the following result:


\begin{theorem}
Let $I=I(p,q)$ be a smooth quantity which is invariant by the solutions of (\ref{system}). Then, when using a PLMM (\ref{eq12}) for which $m=1$, it holds that
\begin{eqnarray}
\lefteqn{I\left( \begin{array}{c} p_n \\ q_n \end{array} \right)- I\left( \begin{array}{c} p_0 \\ q_0 \end{array} \right)} \nonumber \\
&=& \sum_{j=r}^{2r-1} h^j \nabla I \left( \begin{array}{c} p(t_n) \\ q(t_n) \end{array} \right)^T \bigg[  \sum_{i=2}^{k_p'} x_{i,p}^n \left( \begin{array}{c} e_{j,i,pp}(t_n) \\ e_{j,i,qp}(t_n) \end{array} \right)+ \sum_{i=2}^{k_q'} x_{i,q}^n \left( \begin{array}{c} e_{j,i,pq}(t_n) \\ e_{j,i,qq}(t_n) \end{array} \right) \bigg] \nonumber \\
&&- \sum_{j=r}^{2r-1} h^j \int_{t_0}^{t_n} (\nabla I) \left( \begin{array}{c} p(s) \\ q(s) \end{array} \right)^T \left( \begin{array}{c} c_{j,p} p^{(j+1)}(s) \\ c_{j,q} q^{(j+1)} (s) \end{array} \right)ds \nonumber \\
&&+\sum_{j=r}^{2r-1} h^j \nabla I \left( \begin{array}{c} p_0 \\ q_0 \end{array} \right)^T  \left( \begin{array}{c} e_{j,1,p}(t_0) \\ e_{j,1,q}(t_0) \end{array} \right)+ O(h^{2r}). \label{thI}
\end{eqnarray}
\end{theorem}
\begin{proof}
The proof is just based on the decomposition
\begin{eqnarray}
I\left( \begin{array}{c} p_n \\ q_n \end{array} \right)- I\left( \begin{array}{c} p_0 \\ q_0 \end{array} \right)
&=&\bigg[I\left( \begin{array}{c} p_n \\ q_n \end{array} \right)- I\left( \begin{array}{c} p_h(t_n) \\ q_h(t_n) \end{array} \right)\bigg]
+ \bigg[I\left( \begin{array}{c} p_h(t_n) \\ q_h(t_n) \end{array} \right)- I\left( \begin{array}{c} p_h(t_0) \\ q_h(t_0) \end{array} \right)\bigg]\label{nsei}\nonumber \\&& +
\bigg[I\left( \begin{array}{c} p_h(t_0) \\ q_h(t_0) \end{array} \right)-I\left( \begin{array}{c} p_0 \\ q_0 \end{array} \right)\bigg].
\nonumber
\end{eqnarray}
Then, the first term in (\ref{thI}) follows from (\ref{ae}) with $m=1$, the definition (\ref{bcad60}), and the property that,
 when evaluating $\nabla I$, $(p_h(t),q_h(t))^T$ differs from $(p(t),q(t))^T$ in $O(h^r)$ terms. As for the second term, Lemma \ref{Ismooth} can be directly applied. Finally, the last term in (\ref{thI}) is a consequence of (\ref{bcad60}) at $t=t_0$.

\end{proof}

\begin{remark}
Since $x_{i,p}^n$ and $x_{i,q}^n$ have unit modulus, the time behaviour of (\ref{thI}) is determined by the behaviour of the error on the smooth numerical solution taking Lemma \ref{Ismooth} and Theorem \ref{rem_ham_sym} into account, as well as the behaviour of  terms
\begin{eqnarray}
&&\nabla I \left( \begin{array}{c} p(t_n) \\ q(t_n) \end{array} \right)^T \left( \begin{array}{c} e_{j,i,pp}(t_n) \\ e_{j,i,qp}(t_n) \end{array} \right), \quad \nabla I \left( \begin{array}{c} p(t_n) \\ q(t_n) \end{array} \right)^T \left( \begin{array}{c} e_{j,i,pq}(t_n) \\ e_{j,i,qq}(t_n) \end{array} \right).  \nonumber
\end{eqnarray}
\label{rem_inv_time}
\end{remark}

\section{Double pendulum problem}\label{sec5}
The previous results will be applied in this section to explain the numerical integration of the double pendulum problem with PLMMs.
We will scale the problem so that gravity acceleration can be considered equal to $1$ and, in particular,  we will take two masses where $m_1=1$, $m_2=2$ and two inextensible weightless strings of unit length. (For the general case, see \cite{L}, although the conclusions with respect to error growth with time will be the same whenever the oscillations are small enough and the quotient of the associated normal frequencies is not rational.) In our particular case, the kinetic and potencial energies are given respectively by
\begin{eqnarray}
T&=& \frac{1}{2}[ 3 \dot{q}_1^2+2 \dot{q}_2^2+ 4 \dot{q}_1 \dot{q}_2 \cos(q_2-q_1)], \quad V=-(3 \cos(q_1)+2 \cos(q_2)), \nonumber
\end{eqnarray}
where $(q_1,q_2)$ denote the angles which both strings form with the vertical line. As the Lagrangian is $L=T-V$, it is well-known that the associated momentums are
\begin{eqnarray}
p_1&=& L_{\dot{q}_1}=3 \dot{q}_1+ 2 \dot{q}_2\cos(q_2-q_1), \nonumber \\
p_2&=& L_{\dot{q}_2}=2 \dot{q}_2+ 2 \dot{q}_1\cos(q_2-q_1). \label{pdp}
\end{eqnarray}
In such a way, $(p_1 \, p_2)^T= M(q_1,q_2) (\dot{q}_1 \, \dot{q}_2)^T$ with
\begin{eqnarray}
M(q_1,q_2)=\left( \begin{array}{cc} 3 & 2 \cos(q_2-q_1) \\ 2 \cos(q_2-q_1) & 2 \end{array} \right). \nonumber
\end{eqnarray}
and the kinetic energy can be written as
\begin{eqnarray}
T&=& \frac{1}{2} ( \dot{q}_1 \, \dot{q}_2) M(q_1,q_2) \left( \begin{array}{c} \dot{q}_1 \\ \dot{q}_2 \end{array} \right) \nonumber \\
&=& \frac{1}{2} (p_1 \, p_2)^T M(q_1,q_2)^{-1} M(q_1,q_2)  M(q_1,q_2)^{-1} \left( \begin{array}{c} p_1 \\ p_2 \end{array} \right) \nonumber \\
&=& \frac{1}{2} (p_1 \, p_2)^T M(q_1,q_2)^{-1} \left( \begin{array}{c} p_1 \\ p_2 \end{array} \right). \nonumber
\end{eqnarray}
Then, the corresponding differential equations are written as a nonseparable Hamiltonian system with Hamiltonian
\begin{eqnarray}
H=T+V.
\label{ham_dd}
\end{eqnarray}
We will be interested in small oscillations of the double pendulum, which, according to \cite{L}, can be approximated by the solution of the Hamiltonian system where all the terms smaller than those associated with the second power of $q_1,q_2, \dot{q}_1,\dot{q}_2$ are neglected. In this manner, a linear system turns up, which can be exactly solved, giving rise to
\begin{eqnarray}
q_1(t)&=&A \cos( \omega_1 t-\delta_1)+ B \cos( \omega_2 t-\delta_2), \nonumber \\
q_2(t)&=&A c_{+} \cos( \omega_1 t-\delta_1)+ B c_{-} \cos( \omega_2 t-\delta_2), \label{aproxdp}
\end{eqnarray}
where
$$
\omega_1=3(1+\sqrt{\frac{2}{3}}), \quad \omega_2= 3(1-\sqrt{\frac{2}{3}}), \quad c_{+}=-\frac{\sqrt{6}}{2}, \quad c_{-}=\frac{\sqrt{6}}{2},
$$
and the constants $A,B, \delta_1, \delta_2$ are arbitrary constants to be determined by initial conditions.

Then we notice that
\begin{eqnarray}
\dot{q}_1(t)&=&-A \omega_1 \sin (\omega_1 t-\delta_1)- B \omega_2 \sin( \omega_2 t-\delta_2), \nonumber \\
\dot{q}_2(t)&=&-A c_{+} \omega_1 \sin( \omega_1 t-\delta_1)- B c_{-} \omega_2 \sin( \omega_2 t-\delta_2). \label{q12d}
\end{eqnarray}
From (\ref{pdp}) and (\ref{q12d}), we have
\begin{eqnarray}
p_1(t)&=& 3[-A \omega_1 \sin(\omega_1 t-\delta_1)- B \omega_2 \sin(\omega_2 t-\delta_2)] \nonumber \\
&&+2[-A c_{+} \omega_1 \sin(\omega_1 t-\delta_1)- B c_{-}\omega_2 \sin(\omega_2 t-\delta_2)]\cos(\alpha(t)), \nonumber \\
p_2(t)&=& 2[-A c_{+} \omega_1 \sin(\omega_1 t-\delta_1)-B c_{-}\omega_2 \sin( \omega_2 t-\delta_2)] \nonumber \\
&&+2[-A  \omega_1 \sin(\omega_1 t-\delta_1)- B \omega_2 \sin(\omega_2 t-\delta_2)]\cos(\alpha(t)),
\label{paproxdp}
\end{eqnarray}
where
\begin{eqnarray}
\alpha(t)=A(c_{+}-1)\cos(\omega_1 t-\delta_1)+B(c_{-}-1) \cos(\omega_2 t-\delta_2).
\label{alpha}
\end{eqnarray}
Considering then (\ref{aproxdp}) as a good enough approximation of the exact solution, we will justify now how the error in the Hamiltonian associated to the smooth part of the numerical solution should grow when integrating with different types of linear multistep methods.

\subsection{Error associated to the smooth part of the numerical solution}

\subsubsection{Symmetric PLMMs}

First of all, when the method is a PLMM formed by two symmetric LMMs, Theorem \ref{rem_ham_sym} says how the error in the smooth part of the numerical solution behaves. As the time derivatives of $q_1$, $q_2$, $p_1$, $p_2$ are clearly bounded, we must analyse  the term of the form
$$\int_{t_0}^t p^{(k+1)}(s)^T q^{(k+1)}(s), \quad k=r/2,\dots,r-1.$$
To this end, using the second equation in (\ref{difpq}), undoing what has been done for part of the proof of the theorem, we must study whether
$$
\int_{t_0}^t \dot{p}(s)^T q^{(2k+1)}(s)ds= p(s)^T q^{(2k+1)}(s)|_{s=t_0}^{t}-\int_{t_0}^t p(s)^T q^{(2k+2)}(s) ds$$
is bounded with $t$. Obviously, the first term is again bounded. As for the second, we firstly notice that the even derivatives of $q_1$ and $q_2$ always consist of a linear combination of $\cos(\omega_1 t-\delta_1)$ and $\cos(\omega_2 t-\delta_2)$. Then, from  (\ref{paproxdp}), what we must study is the behaviour of
\begin{eqnarray}
&&\int_{t_0}^t \cos(\omega_i s- \delta_i) \sin(\omega_j s- \delta_j)ds, \nonumber \\
&&\int_{t_0}^t \cos(\omega_i s- \delta_i) \sin(\omega_j s- \delta_j)\cos(\alpha(s))ds, \quad i,j \in \{1,2\}. \label{intdp}
\end{eqnarray}
where, according to (\ref{alpha}), $\alpha(s)$ oscillates between $-|A(c_{+}-1)|-|B(c_{-}-1)|$ and $|A(c_{+}-1)|+|B(c_{-}-1)|$ but in a non-periodic way because $\omega_1/\omega_2$ is not rational. Using trigonometric identities, the first integral in (\ref{intdp}) can be written as
$$
\frac{1}{2}\int_{t_0}^t [\sin((\omega_i+\omega_j)s-\delta_i-\delta_j)-\sin((\omega_i-\omega_j)s+\delta_j-\delta_i))]ds,$$
which is obviously bounded with $t$. As for the second line in (\ref{intdp}), it is the integral of the same sine functions as those of the first one (which, at a period, determine the same area above and below the $s$-axis) but modulated in amplitude by another function which also oscillates with $s$. The fact that the frequency of that oscillation in amplitude is not related through a rational number to the frequencies $(\omega_i+\omega_j)$ and $(\omega_i-\omega_j)$ of the sine functions explains that, when $t$ grows, the integral is also  bounded.

\subsubsection{Non-symmetric PLMMs}

In the non-symmetric case and according to Corollary \ref{corol_ham}, the growth with time of the coefficients associated to the even powers of $h$ will behave in the same way as those for symmetric methods. However, when the methods are not symmetric, as $c_{j,q}$ and $c_{j,p}$ do not necessarily vanish for odd $j$, we must also look at the growth with time of the following expressions  which multiply the odd powers of $h$:
$$
\int_{t_0}^t \dot{p}(t)^T q^{(2k)}(t)dt, \quad \int_{t_0}^t \dot{q}(t)^T p^{(2k)}(t)dt, \qquad k=[r/2]+1,\dots,r.$$
By integration by parts,
\begin{eqnarray}
\int_{t_0}^t \dot{p}(t)^T q^{(2k)}(t)dt&=& p(t)^T q^{(2k)}(t)|_{t=t_0}^t -\int_{t_0}^t p(t)^T q^{(2k+1)}(t)dt, \nonumber \\
\int_{t_0}^t \dot{q}(t)^T p^{(2k)}(t) dt&=&\sum_{l=1}^{2k} (-1)^{l+1} q^{(l)}(t)^T p^{(2k-l)}(t)|_{t=t_0}^t+\int_{t_0}^t q^{(2k+1)}(t)^T p(t)dt. \nonumber
\end{eqnarray}
As the derivatives of $q_1,q_2,p_1,p_2$ are bounded with time, we only must look at the growth with time of
$$\int_{t_0}^t q^{(2k+1)}(t)^T p(t)dt.$$
From (\ref{aproxdp}), we can see that the odd derivatives of $q_1$ and $q_2$  consist of linear combinations of $\sin(\omega_i t-\delta_i)$ and $\sin(\omega_j t-\delta_j)$. From (\ref{paproxdp}), the previous integral will contain expressions of the form
$$
\int_{t_0}^t \sin^2 (\omega_i t-\delta_i)dt, \quad i=1,2,
$$ which obviously grow linearly with time because $\sin^2 (\omega_i t-\delta_i)$ is a positive and periodic function with period $\pi/\omega_i$. Thus, the integral is linearly additive at each period.

\subsection{Error associated to the non-common roots of the first characteristic polynomials of unit modulus}

 From Remark \ref{rem_inv_time} and considering that $\nabla H(p,q)=(\dot{q},-\dot{p})$, which is bounded in our problem, what we must study is the growth with time of $e_{j,i,pp}$, $e_{j,i,pq}$, $e_{j,i,qq}$, $e_{j,i,qp}$. In order to  apply Theorem \ref{th3}, we require that the derivatives of $f$ and $g$ at $(p(t),q(t))$ and their time derivatives are bounded with time and that the transition matrices associated to (\ref{eripp}) and (\ref{eriqq}) are also bounded.
Taking into account that, in our problem,
\begin{eqnarray}
H(p_1,p_2,q_1,q_2)&=&\frac{1}{2(6-4 \cos^2(q_2-q_1))}(2p_1^2-4 \cos(q_2-q_1)p_1 p_2+3 p_2^2)
\nonumber \\
&&-3 \cos(q_1)-2 \cos(q_2), \label{hamdp}
\end{eqnarray}
$f=-\nabla_q H$, $g=\nabla_p H$,  it is clear that the components of the Jacobian of $(f,g)^T$ at $(p_1(t),p_2(t),q_1(t),q_2(t))$ and their time derivatives are bounded. Let us see what happens with the transition matrices of (\ref{eripp}) and (\ref{eriqq}). We firstly notice that, after some calculations,
\begin{eqnarray}
f_p(p(t),q(t))=\left( \begin{array}{cc} -a(t) & -b(t) \\ a(t) & b(t) \end{array} \right), \quad g_q(p(t),q(t))=\left( \begin{array}{cc} a(t) & -a(t) \\ b(t) & -b(t) \end{array} \right),
\label{f_pg_q}
\end{eqnarray}
where
\begin{eqnarray}
a(t)&=& \frac{2\sin(2 \alpha(t))}{[1+2\sin^2(\alpha(t))]^2}p_1(t)-\frac{\sin(\alpha(t))[5-2 \sin^2(\alpha(t))]}{[1+2\sin^2(\alpha(t))]^2}p_2(t) \nonumber \\
&=&\frac{\sin(2 \alpha(t))}{1+2\sin^2(\alpha(t))}\dot{q}_1(t)-\frac{2 \sin(\alpha(t))}{1+2\sin^2(\alpha(t))}\dot{q}_2(t), \nonumber \\
b(t)&=&\frac{3 \sin(2 \alpha(t))}{[1+2\sin^2(\alpha(t))]^2}p_2(t)-\frac{\sin(\alpha(t))[5-2 \sin^2(\alpha(t))]}{[1+2\sin^2(\alpha(t))]^2}p_1(t) \nonumber \\
&=&-\frac{3\sin(\alpha(t))}{1+2\sin^2(\alpha(t))}\dot{q}_1(t)+\frac{\sin(2\alpha(t))}{1+2\sin^2(\alpha(t))}\dot{q}_2(t). \label{ab}
\end{eqnarray}
Here, for the second and fourth equalities, (\ref{pdp}) has been used and $\alpha(t)$ is the function in (\ref{alpha}). As it is well-known that the transition matrix $\Phi(t,t_0)$ associated to a linear differential system
$$
\dot{e}(t)=A(t) e(t)$$
satisfies that
$$
\mbox{det}(\Phi(t,t_0))=e^{ \int_{t_0}^t \mbox{tr}(A(\tau))d\tau},$$
then, from (\ref{f_pg_q}), the determinants of the  transition matrices associated to (\ref{eripp}) and (\ref{eriqq}) are
 \begin{eqnarray}
 e^{ \lambda_{p,i,p} \int_{t_0}^t [b(\tau)-a(\tau)]d \tau}, \quad  e^{ -\lambda_{q,i,q} \int_{t_0}^t [b(\tau)-a(\tau)]d \tau},
 \label{exponenciales}
 \end{eqnarray}
 respectively. Now, from (\ref{ab}),
 $$
 b(t)-a(t)=\frac{\sin(2 \alpha(t))+2 \sin(\alpha(t))}{1+2 \sin^2(\alpha(t))}\dot{q}_2(t)-\frac{\sin(2 \alpha(t))+3 \sin(\alpha(t))}{1+2 \sin^2(\alpha(t))}\dot{q}_1(t).$$
 Consequently, the integral in  expressions (\ref{exponenciales}) is bounded with time $t$ since the integral of $\dot{q}_1(\tau)$ and $\dot{q}_2(\tau)$ in (\ref{q12d}) is obviously bounded and what we have are these oscillating functions
modulated in amplitude by other ones which also oscillate in a bounded way
with time, but in an erratic way with respect to the first ones since $\omega_1/\omega_2$ is irrational. This implies that, independently of the sign of $\lambda_{p,i,p}$ and $\lambda_{q,i,q}$, the determinant of the transition matrices associated to (\ref{eripp}) and (\ref{eriqq}) is bounded and far enough from zero when $t$ grows. Therefore, because of the geometric meaning of the determinant, whenever we take initial conditions for (\ref{eripp}) and (\ref{eriqq}) in a bounded set, the area of the set formed by the solutions of those systems at time $t$ does not grow with time neither goes to zero.

On the other hand, it is easily seen that $0$ and $\pm (b(t)-a(t))$ are respectively the eigenvalues of the matrices in (\ref{f_pg_q}). Then, for the system
$$
\dot{e}(t)= \lambda_{p,i,p} f_p(p(t),q(t)) e(t),$$
as $\mbox{ker}(f_p(p(t),q(t))-(b(t)-a(t))I)=\mbox{span}\{ (1 \, \, -1)^T \} $, which is independent of time, it happens that
$$
e^{ \lambda_{p,i,p} \int_{t_0}^t [b(\tau)-a(\tau)]d \tau} \left( \begin{array}{c} 1 \\ -1 \end{array} \right)$$
is a solution and, with the same argument as for the determinant, it is bounded with time. If we take another initial condition which is linearly independent with $(1 \,\, -1)^T$, it will form with the origin and $(1 \, \, -1)^T$ a paralelogram, and the paralelogram which corresponds to advance a time $t$ from the initial one will have a bounded (away from zero) area according to what has been argued for the determinant of the transition matrix. Because of this, the solution after time $t$ corresponding to this latter initial condition must also be bounded. This completes the proof of bounded behaviour of the transition matrix of this system. On the other hand, for the system
$$
\dot{e}(t)= \lambda_{q,i,q} g_q(p(t),q(t)) e(t),$$
as $\mbox{ker}(g_q(p(t),q(t)))=\mbox{span}\{(1 \, \, \,\,1)^T \}$, which is independent of time, it happens that
$$
\left( \begin{array}{c} 1 \\ 1 \end{array} \right) $$
is a solution, and it is in fact constant with time. With the same argument as before, if we take another initial condition which is linearly independent with this one, it will form with the origin and $(1 \, \, \,\,1)^T$ a paralelogram, and the paralelogram which corresponds to advance a time $t$ will have a bounded (away from zero) area according to what has been argued for the determinant of the transition matrix. Thus, the solution after time $t$ corresponding to this new initial condition must also be bounded. This implies that the second transition matrix is also bounded with time.

\subsection{Numerical experiments}

The results of the previous section, concerning the numerical approximation of the double pendulum problem with PLMMs, will be illustrated here. To this end, several LMMs will be used to study the time behaviour of the error in the Hamiltonian (\ref{ham_dd}), namely:

\begin{enumerate}
\item[1.]
The symmetric PLMM of second order (called PLMM2, cf. \cite{CH} )
\begin{eqnarray}
\rho_p(x)&=&(x-1)(x+1), \quad \sigma_p(x)=2x, \nonumber \\
\rho_q(x)&=&(x-1)(x^2+1), \quad \sigma_q(x)=x^2+x. \label{plmm2}
\end{eqnarray}
\item[2.]
The symmetric, nonpartitioned LMM
\begin{eqnarray}
\rho_p(x)=\rho_q(x)=(x-1)(x+1), \quad \sigma_p(x)=\sigma_q(x)=2x. \label{lmm2}
\end{eqnarray}
\item[3.]
The nonsymmetric, nonpartitioned third-order Adams method \cite{HNW}
\begin{eqnarray}
\rho_p(x)=\rho_q(x)=x^2(x-1), \quad \sigma_p(x)=\sigma_q(x)=\frac{23}{12}x^2-\frac{16}{12}x+\frac{5}{12}. \label{adams3}
\end{eqnarray}
\item[4.]
A nonsymmetric PLMM where one of the methods is the symmetric LMM in (\ref{lmm2}) and the other one is the second-order Adams method.
More specifically,
\begin{eqnarray}
\rho_p(x)&=&(x-1)(x+1), \quad \sigma_p(x)=2x, \nonumber \\
\rho_q(x)&=&x(x-1), \quad \sigma_q(x)=\frac{3}{2}x-\frac{1}{2}. \label{sim_nosim}
\end{eqnarray}
\end{enumerate}
For all numerical experiments, we have considered as initial conditions
$$p_1(0)=0, \quad p_2(0)=0, \quad q_1(0)=\frac{\pi}{12}, \quad q_2(0)=\frac{\pi}{6}.$$
Moreover, we have taken as starting procedure the \lq exact' starting values, i.e. those calculated with ode45 subroutine in Matlab with an error tolerance of $10^{-13}$.

We firstly  notice that in (\ref{plmm2}) $\rho_p$ and $\rho_q$ have no common roots except for $x_1=1$. In Figure \ref{fig1}, we represent the error in the Hamiltonian (\ref{hamdp}) against time. We can clearly observe that, as expected, the error in the Hamiltonian is bounded till very long times.
\begin{figure*}
\centerline{\includegraphics[width=100mm]{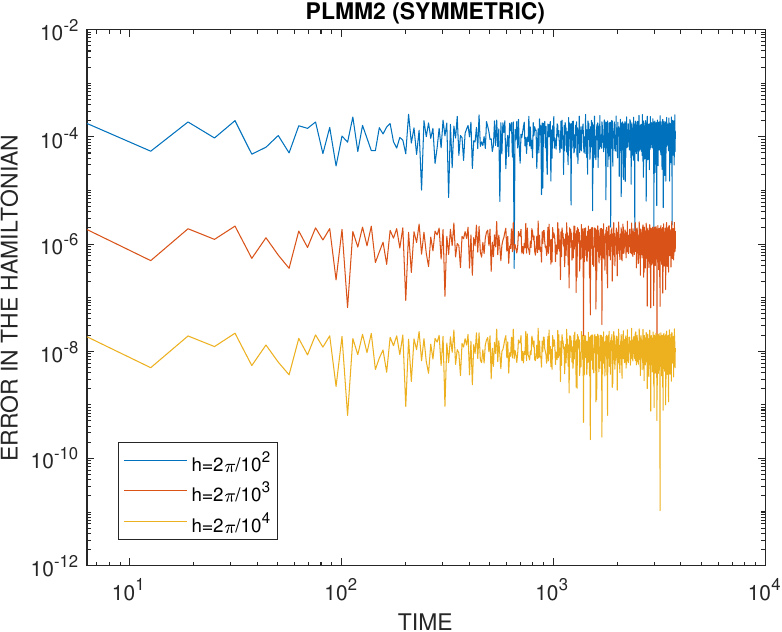}}
\caption{Error in the Hamiltonian against time measured at integer multiples of $2\pi$ when integrating the double pendulum problem with symmetric PLMM2 (\ref{plmm2}). }
\label{fig1}
\end{figure*}

However, when we consider the symmetric non-partitioned LMM (\ref{lmm2}),
an exponential error growth in the Hamiltonian with time turns up, as can be observed in Figure \ref{fig2}.
\begin{figure*}
\centerline{\includegraphics[width=100mm]{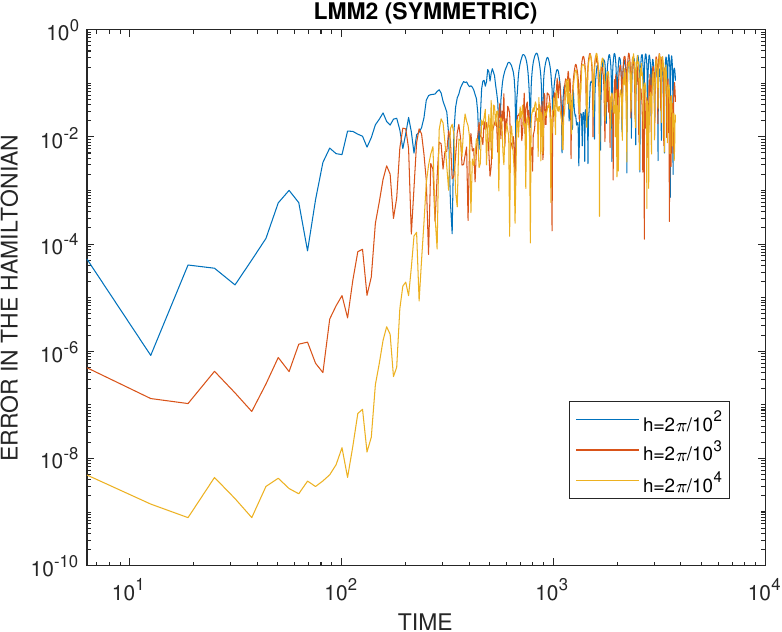}}
\caption{Error in the Hamiltonian against time measured at integer multiples of $2\pi$ when integrating the double pendulum problem with the symmetric non-partitioned LMM (\ref{lmm2}). }
\label{fig2}
\end{figure*}

On the other hand, if we take a non-symmetric non-partitioned LMM, as Adams method, the coefficients of the error in the Hamiltonian associated to the root $1$ on the even powers of $h$ will also be bounded with time. However, those associated to the odd powers of $h$ will grow linearly with time. Because of that, as the first characteristic polynomial of Adams method just has the root $x_1=1$ of unit of modulus, if the order is odd, linear error growth with time is expected, as can be seen in Figure \ref{fig3} for
(\ref{adams3}).
\begin{figure*}
\centerline{\includegraphics[width=100mm]{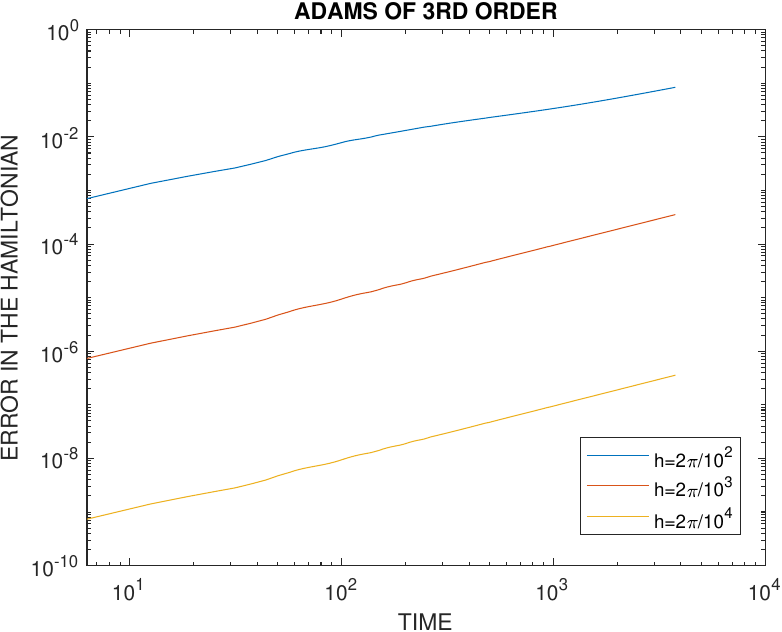}}
\caption{Error in the Hamiltonian against time measured at integer multiples of $2\pi$ when integrating the double pendulum problem with Adams method of 3rd order (\ref{adams3}). }
\label{fig3}
\end{figure*}

Finally, as for (\ref{sim_nosim}), notice that the first characteristic polynomials do not have common roots of unit modulus except for $x_1=1$, and the coefficient of the error associated to the root $x_{2,p}=-1$ of $\rho_p$ is known to be bounded with time, as stated before. As for the error in the Hamiltonian associated to the root $x_1=1$, the coefficient multiplying $h^2$ will be bounded, but the one multiplying $h^3$ will grow linearly. Because of that, when $h$ is small enough, the error seems to be bounded with time at the beginning but when $t$ grows the error associated to $h^3$ dominates and linear error growth is observed. That can be confirmed in Figure \ref{fig4}.
\begin{figure*}
\centerline{\includegraphics[width=100mm]{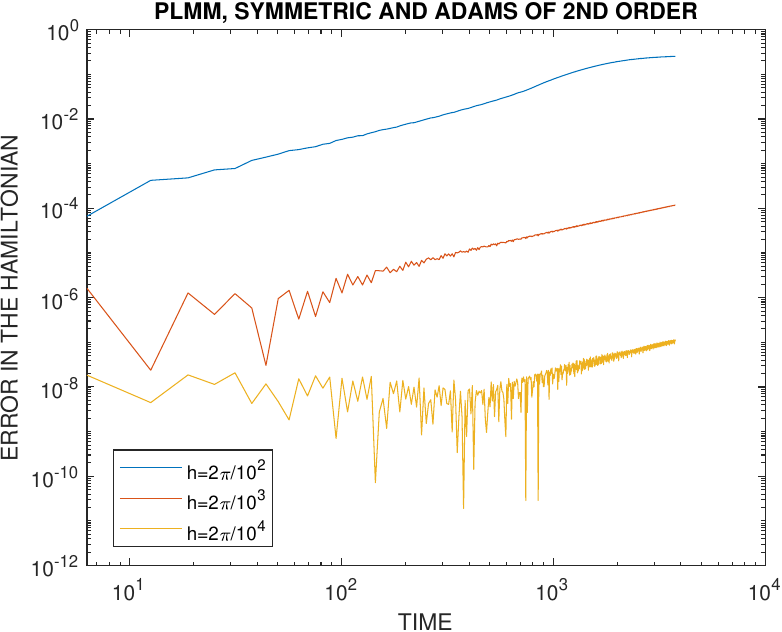}}
\caption{Error in the Hamiltonian against time measured at integer multiples of $2\pi$ when integrating the double pendulum problem with PLMM (\ref{sim_nosim}). }
\label{fig4}
\end{figure*}

As a conclusion, the error in the Hamiltonian with the symmetric PLMM behaves in a much more favourable manner.

\section{Concluding remarks}
The main contributions of this manuscript are thus the following:
\begin{list}{$\bullet$}{}
\item
An asymptotic expansion on the stepsize is given for the global error when integrating any $n$-dimensional ODE system ($n\ge 2$) with a PLMM. When some information on the exact solution to be approximated is known, this may allow to know how the coefficients in that expansion grow with time, as well as the behaviour with respect to the possible conservation of invariants.
According to this analysis, the following conclusions hold.
\item
It is not recommended in principle to consider PLMMs where the first characteristic polynomials of the method have common roots different from $x_1=1$, mainly if these roots have unit modulus, since they seem to lead to an exponential growth of error with time.
\item
Symmetric PLMMs with no common roots except $x_1=1$ in their first characteristic polynomials turn out very efficient  since the annihilation of the coefficients associated to the even powers of the stepsize in the local truncation error simplify the expressions for the coefficients associated to the global error and also those appearing for the error in the invariants. Moreover, although a non-smooth part turns up in the numerical solution in such a case, the analysis in the present paper provides assumptions under which that part is under control, even in the case of non-separable problems. Furthermore, the fact that these methods can be explicit enables them as a very valuable tool from the computational point of view compared with other geometric integrators for general non-separable problems with a certain structure.
\item
A detailed analysis has been performed for the  case study of small oscillations of the double pendulum, justifying in such a way the advantageous behaviour of symmetric PLMMs with respect to error growth in the Hamiltonian against other explicit LMMs, either symmetric and non-partitioned, non-symmetric and partitioned and non-symmetric and non-partitioned.
\end{list}

\bmhead{Acknowledgements}
This research has been supported by Ministerio de Ciencia e Innovaci\'on project PID2023-147073NB-I00.

\begin{appendices}

\section{Proof of Theorem \ref{th1}}\label{secA1}
For the sake of simplicity, we will assume that all roots of $\rho_p$ and $\rho_q$ are single and non-zero. The more general case is considered in \cite{R}.

What we must prove is which conditions the coefficients $e_{j,i,\alpha}$, $e_{j,i,\alpha \beta}$ must satisfy so that (\ref{ae}) holds. For that, we consider
\begin{eqnarray}
\varepsilon_n&=&p_n-p(t_n)-\sum_{j=r}^{2r-1} h^j \big[\sum_{i=1}^m x_i^n e_{j,i,p}(t_n) +\sum_{i=m+1}^{k_p} x_{i,p}^n e_{j,i,pp}(t_n)+ \sum_{i=m+1}^{k_q} x_{i,q}^n e_{j,i,pq}(t_n)\big], \nonumber \\
\eta_n&=&q_n-q(t_n)-\sum_{j=r}^{2r-1} h^j \big[\sum_{i=1}^m x_i^n e_{j,i,q}(t_n) +\sum_{i=m+1}^{k_p} x_{i,p}^n e_{j,i,qp}(t_n)+ \sum_{i=m+1}^{k_q} x_{i,q}^n e_{j,i,qq}(t_n)\big], \nonumber
\nonumber
\end{eqnarray}
and we want to find conditions under which can assure that $\varepsilon_n$ and $\eta_n$ are $O(h^{2r})$.

We notice that
\begin{eqnarray}
0&=&\rho_p(E)p_n-h \sigma_p(E)f(p_n,q_n) \nonumber \\
&=&\rho_p(E)p(t_n)-h \sigma_p(E)f(p(t_n),q(t_n)) \nonumber \\
&&+\sum_{j=r}^{2r-1} h^j \bigg[ \sum_{i=1}^m \rho_p(E) x_i^n e_{j,i,p}(t_n)+\sum_{i=m+1}^{k_p} \rho_p(E) x_{i,p}^n e_{j,i,pp}(t_n)
\nonumber \\
&&\hspace{1.5cm}
+\sum_{i=m+1}^{k_q}  \rho_p(E) x_{i,q}^n e_{j,i,pq}(t_n)
\nonumber \\
&&\hspace{1.5cm} -h \big[ \sum_{i=1}^m \sigma_p(E)  x_i^n f_p(p(t_n),q(t_n)) e_{j,i,p}(t_n)
\nonumber \\
&&\hspace{2cm}+\sum_{i=m+1}^{k_p} \sigma_p(E) x_{i,p}^n f_p(p(t_n),q(t_n)) e_{j,i,pp}(t_n)
\nonumber \\
&&\hspace{2cm}
+\sum_{i=m+1}^{k_q} \sigma_p(E) x_{i,q}^n f_p(p(t_n),q(t_n)) e_{j,i,pq}(t_n) \nonumber \\
&& \hspace{2cm}
+\sum_{i=1}^m \sigma_p(E)  x_i^n f_q(p(t_n),q(t_n)) e_{j,i,q}(t_n)
\nonumber \\
&& \hspace{2cm}
+\sum_{i=m+1}^{k_p} \sigma_p(E) x_{i,p}^n f_q(p(t_n),q(t_n)) e_{j,i,qp}(t_n)
\nonumber \\
&&\hspace{2cm}+ \sum_{i=m+1}^{k_q} \sigma_p(E) x_{i,q}^n f_q(p(t_n),q(t_n)) e_{j,i,qq}(t_n)
\big] \bigg] \nonumber \\
&& + \rho_p(E) \varepsilon_n-h \sigma_p(E)[f_p (p(t_n),q(t_n))\varepsilon_n+f_q (p(t_n),q(t_n))\eta_n] \nonumber \\
&&+O\big(h (\|\varepsilon_n\|^2+\|\eta_n\|^2)\big)+O(h^{2r+1}),
\label{for1}
\end{eqnarray}
and something similar for the second equation.
Then, by using (\ref{LTE}) and the notation in (\ref{bcad0b}), (\ref{bcad0c}), it follows that (\ref{for1}) is equivalent to
\begin{eqnarray}
0&=&\sum_{j=r}^{2r-1} h^j \bigg[ \rho_p(E)e_{j,1,p}(t_n)-h \sigma_p(E) [f_p(p(t_n),q(t_n))e_{j,1,p}(t_n) \nonumber \\
&& \hspace{5cm} +f_q(p(t_n),q(t_n))e_{j,1,q}(t_n)-c_{j,p}p^{(j+1)}(t_n)] \nonumber \\
&&\hspace{1cm}+\sum_{i=2}^m x_i^n \big[\rho_{p,i}(E) e_{j,i,p}(t_n)-h \sigma_{p,i}(E)[f_p(p(t_n),q(t_n))e_{j,i,p}(t_n) \nonumber \\
&& \hspace{7cm}+f_q(p(t_n),q(t_n))e_{j,i,q}(t_n)]\big] \nonumber \\
&&\hspace{1cm}+\sum_{i=m+1}^{k_p} x_{i,p}^n \big[\rho_{p,i,p}(E) e_{j,i,pp}(t_n)-h \sigma_{p,i,p}(E)[f_p(p(t_n),q(t_n))e_{j,i,pp}(t_n) \nonumber \\
&& \hspace{7cm}+f_q(p(t_n),q(t_n))e_{j,i,qp}(t_n)]\big] \nonumber \\
&&\hspace{1cm}+\sum_{i=m+1}^{k_q} x_{i,q}^n \big[\rho_{p,i,q}(E) e_{j,i,pq}(t_n)-h \sigma_{p,i,q}(E)[f_p(p(t_n),q(t_n))e_{j,i,pq}(t_n) \nonumber \\
&& \hspace{7cm}+f_q(p(t_n),q(t_n))e_{j,i,qq}(t_n)]\big]\bigg] \nonumber \\
&& + \rho_p(E) \varepsilon_n-h \sigma_p(E)[f_p (p(t_n),q(t_n))\varepsilon_n+f_q (p(t_n),q(t_n))\eta_n] \nonumber \\
&&+O(h^{2r+1}). \label{formulon}
\end{eqnarray}
We would like that
\begin{eqnarray}
\rho_p(E) \varepsilon_n-h \sigma_p(E)[f_p (p(t_n),q(t_n))\varepsilon_n+f_q (p(t_n),q(t_n))\eta_n]=O(h^{2r+1}), \nonumber \\
\rho_q(E) \eta_n-h \sigma_q(E)[g_p (p(t_n),q(t_n))\varepsilon_n+g_q (p(t_n),q(t_n))\eta_n]=O(h^{2r+1}), \label{epseta}
\end{eqnarray}
in order to deduce from here that $\varepsilon_n=O(h^{2r})$ and $\eta_n=O(h^{2r})$ after imposing that $\varepsilon_\nu=O(h^{2r})$ for $\nu=0,1,\dots,k_p-1$ and $\eta_\nu=O(h^{2r})$ for $\nu=0,1,\dots,k_q-1$. (For that deduction, we have to apply a result very similar to Lemma 5.6 in \cite{H}, and which proof can be seen in \cite{R}.)

We will focus on imposing that the first equation in (\ref{epseta}) is satisfied, because that is the one written in (\ref{formulon}), but an analogous argument is also valid to impose that the second equation is satisfied. In order that the terms associated to the root $1$ in (\ref{formulon}) are $O(h^{2r+1})$, it suffices to impose (\ref{ej1}), taking into account that the method $(\rho_p,\sigma_p)$ is  assumed to have order $r$. On the other hand, for the terms associated to $x_i^n$ ($i=2,\dots,m$), considering (\ref{LTEb}), it suffices to impose (\ref{eji}).

Let us now proceed to see how  the rest of terms should behave.
\begin{list}{$\bullet$}{}
\item
In order that the term in $h^{r}x_{i,q}^n$ in (\ref{formulon}) vanishes, as $\rho_p(x_{i,q})\neq 0$, $\rho_{p,i,q}(1)\neq 0$ and using also that, in an analogous way, $\rho_{q,i,p}(1)\neq 0$, (\ref{erpqqp}) should hold.
\item
Looking then at the term in $h^{r+1} x_{i,p}^n$ in the same sum, it should happen that
$$
\rho_{p,i,p}'(1)  \dot{e}_{r,i,pp}(t)-\sigma_{p,i,p}(1) f_p(p(t),q(t))e_{r,i,pp}(t)=0,$$
which is equivalent to (\ref{eripp}).

Looking then at $h^{r+1} x_{i,q}^n$
in (\ref{formulon}), it should happen that
$$
\rho_{p,i,q}(1)e_{r+1,i,pq}(t_n)-\sigma_{p,i,q}(1) f_q(p(t_n),q(t_n)) e_{r,i,qq}(t_n)=0,$$
which is equivalent to (\ref{er1pq}).
\item
Looking at the term in $h^{r+2} x_{i,p}^n$ in (\ref{formulon}), using abbreviated notation and (\ref{erpqqp}), it is required that
\begin{eqnarray}
\frac{1}{2} \lefteqn{[x_{i,p}^2\rho_p''(x_{i,p})+x_{i,p}\rho_p'(x_{i,p})]\ddot{e}_{r,i,pp}-x_{i,p}\sigma_p'(x_{i,p})\frac{d}{dt}[f_p(p,q)e_{r,i,pp}]}
\nonumber \\
&& +x_{i,p}\rho_p'(x_{i,p})\dot{e}_{r+1,i,pp}-\sigma_p(x_{i,p})[f_p(p,q)e_{r+1,i,pp}+f_q(p,q) e_{r+1,i,qp}]=0, \nonumber
\end{eqnarray}
which is equivalent to (\ref{er1pp}).

Looking now at the term in $h^{r+2} x_{i,q}^n$ in (\ref{formulon}), it follows that
\begin{eqnarray}
\lefteqn{\hspace{-4cm}\rho_p(x_{i,q})e_{r+2,i,pq}+x_{i,q}\rho_p'(x_{i,q})\dot{e}_{r+1,i,pq}-\sigma_p(x_{i,q})[f_p(p,q)e_{r+1,i,pq}+f_q(p,q)e_{r+1,i,qq}]}
\nonumber \\
&&-x_{i,q}\sigma_p'(x_{i,q})\frac{d}{dt}[f_q(p,q) e_{r,i,qq}]=0, \nonumber
\end{eqnarray}
which is equivalent, using (\ref{er1pq}), to (\ref{er2pq}).
\item
Looking now at the term in $h^{r+3} x_{i,p}^n$ in (\ref{formulon}), it follows that
\begin{eqnarray}
\lefteqn{\hspace{-1.4cm}\frac{1}{6}[x_{i,p}^3 \rho_p'''(x_{i,p})+3 \rho_p''(x_{i,p})+x_{i,p} \rho_p'(x_{i,p})]\stackrel{\dots}{e}_{r,i,pp}+\frac{1}{2}[x_{i,p}^2 \rho_p''(x_{i,p})+x_{i,p}\rho_p'(x_{i,p})]\ddot{e}_{r+1,i,pp}} \nonumber \\
&&+x_{i,p}\rho_p'(x_{i,p})\dot{e}_{r+2,i,pp}-\frac{1}{2}[x_{i,p}^2 \sigma_p''(x_{i,p})+x_{i,p}\sigma_p'(x_{i,p})] \frac{d^2}{d t^2}[f_p(p,q) e_{r,i,pp}]\nonumber \\ &&-\sigma_p'(x_{i,p})x_{i,p}\frac{d}{dt}[f_p(p,q) e_{r+1,i,pp}+f_q e_{r+1,i,qp}]
\nonumber \\
&&-\sigma_p(x_{i,p})[f_p(p,q) e_{r+2,i,pp}+f_q(p,q) e_{r+2,i,qp}]=0, \nonumber
\end{eqnarray}
which is equivalent to (\ref{er2ipp}) taking (\ref{eripp}) and (\ref{er1pp}) into account.
\end{list}

Let us assume now that the starting values in (\ref{bcad0a}) are such that
\begin{eqnarray}
p_\nu&=&p(t_\nu)+\sum_{j=r}^{2r-1} s_{\nu,p}^{(j)} h^j+O(h^{2r}), \quad \nu=0,1,\dots,k_p-1,\nonumber \\
q_\nu&=&q(t_\nu)+\sum_{j=r}^{2r-1} s_{\nu,q}^{(j)} h^j+O(h^{2r}), \quad \nu=0,1,\dots,k_q-1.\nonumber
\end{eqnarray}
Then, using that $e_{r,i,pq}=e_{r,i,qp}=0$, for $\epsilon_\nu$ and $\eta_\nu$ to be $O(h^{r+1})$ for the above values of $\nu$, it should happen that
\begin{eqnarray}
s_{\nu,p}^{(r)}&=&\sum_{i=1}^m x_i^{\nu} e_{r,i,p}(t_0)+\sum_{i=m+1}^{k_p} x_{i,p}^\nu e_{r,i,pp}(t_0), \quad \nu=0,1,\dots,k_p-1,\nonumber \\
s_{\nu,q}^{(r)}&=&\sum_{i=1}^m x_i^{\nu} e_{r,i,q}(t_0)+\sum_{i=m+1}^{k_q} x_{i,q}^\nu e_{r,i,qq}(t_0), \quad \nu=0,1,\dots,k_q-1.\label{svalues}
\end{eqnarray}
These are two Vandermonde systems which completely determine the values $e_{r,i,p}(t_0),  e_{r,i,q}(t_0)$ for $i=1,\dots,m$ and
$e_{r,i,pp}(t_0)$ for $i=m+1,\dots,k_p$ and $e_{r,i,qq}(t_0)$ for $i=m+1,\dots,k_q$.
On the other hand, for $\epsilon_\nu$ and $\eta_\nu$ to be $O(h^{r+2})$, there should hold that
\begin{eqnarray}
s_{\nu,p}^{(r+1)}&=&\sum_{i=1}^m x_i^\nu \nu \dot{e}_{r,i,p}(t_0)+\sum_{i=1}^m x_i^\nu e_{r+1,i,p}(t_0)
\nonumber \\
&&+\sum_{i=m+1}^{k_p} x_{i,p}^\nu [ e_{r+1,i,pp}(t_0)+\nu \dot{e}_{r,i,pp}(t_0)]+\sum_{i=m+1}^{k_q} x_{i,q}^\nu e_{r+1,i,pq}(t_0),\nonumber \\
&& \hspace{6cm} \nu=0,1,\dots,k_p-1, \nonumber \\
s_{\nu,q}^{(r+1)}&=&\sum_{i=1}^m x_i^\nu \nu \dot{e}_{r,i,q}(t_0)+\sum_{i=1}^m x_i^\nu e_{r+1,i,q}(t_0)
\nonumber \\
&&+\sum_{i=m+1}^{k_q} x_{i,q}^\nu [e_{r+1,i,qq}(t_0)+\nu \dot{e}_{r,i,qq}(t_0)]+\sum_{i=m+1}^{k_p} x_{i,p}^\nu e_{r+1,i,qp}(t_0), \nonumber \\
&& \hspace{6cm} \nu=0,1,\dots,k_q-1.\label{svalues2}
\end{eqnarray}
We notice that these are again two Vandermonde systems. The first one in the unknowns $e_{r+1,i,p}(t_0)$ $(i=1\dots,m)$,  $e_{r+1,i,pp}(t_0)$ $i=m+1, \dots, m+k_p$  and the second one in $e_{r+1,i,q}(t_0)$ $(i=1,\dots,m)$, $e_{r+1,i,qq}(t_0) (i=m+1,\dots,m+k_q)$. (We notice that the rest of terms can be calculated from values which are already determined through (\ref{ej1}), (\ref{eji}) for $j=r$ and (\ref{er1pq}) and (\ref{er1qp}) for $j=r+1$.)

Proceeding inductively, the initial conditions for the differential systems associated to the error coefficients corresponding to higher powers of $h$ can be determined.

Moreover, when the starting procedure is of order $r+1$, it is clear that
$$s_{\nu,p}^{(r)}=0, \quad \nu=0,1,\dots,k_p-1, \quad s_{\nu,q}^{(r)}=0, \quad \nu=0,1,\dots,k_q-1.$$
Thus, the systems in (\ref{svalues}) are homogeneous and then $e_{r,i,pp}(t_0)=e_{r,i,qq}(t_0)=0$. However, the fact that $$s_{\nu,p}^{(r+1)}=s_{\nu,q}^{(r+1)}=0$$
does not make (\ref{svalues2}) homogeneous because $\dot{e}_{r,i,p}(t_0)$ and $\dot{e}_{r,i,q}(t_0)$ do not vanish in general.

Finally, in the asymptotic expansion, the terms associated to the $n$th-powers of $x_i$, $x_{i,p}$ or $x_{i,q}$ of modulus $<1$ can be dropped off since, when $n>0$, those $n$th-powers are $O(h^{2r})$ for fixed $t_n=t_0+nh$.

\end{appendices}

\end{document}